\documentclass[12pt]{amsart}
\usepackage{graphicx}
\usepackage{leqno,color,psfrag,epsfig,url,hhline
}
\usepackage[dvipsnames]{xcolor}
\numberwithin{figure}{section}
\numberwithin{table}{section}
\newtheorem{theorem}{Theorem}[section]
\newtheorem*{thrm}{Main Theorem}
\newtheorem{lemma}[theorem]{Lemma}
\newtheorem{prop}[theorem]{Proposition}

\theoremstyle{definition}
\newtheorem{definition}[theorem]{Definition}
\newtheorem{example}[theorem]{Example}

\newtheorem{cor}[theorem]{Corollary}

\theoremstyle{remark}
\newtheorem{remark}[theorem]{Remark}

\numberwithin{equation}{section}
\newfont{\tap}{tap scaled 650}

\def \H{{\mathbb H}}

\def \R{{\mathbb R}}

\def \Z{{\mathbb Z}}

\def \[{[ }
\def \]{] }

\def \val{val \ }

\def \val{\mathrm{val\,}}

\def \a{\alpha}
\def \b{\beta}
\definecolor{dgreen}{rgb}{0,0.5,0}        
\definecolor{dred}{rgb}{0.5,0,0}

\begin{document}

\title{Friezes from surfaces and Farey triangulation}
\author{Anna Felikson,  Pavel Tumarkin}
\address{Department of Mathematical Sciences, Durham University, Mathematical Sciences \& Computer Science Building, Upper Mountjoy Campus, Stockton Road, Durham, DH1 3LE, UK}
\email{anna.felikson@durham.ac.uk, pavel.tumarkin@durham.ac.uk}

\begin{abstract}
  We provide a classification of positive integral friezes on marked bordered surfaces in the style of Conway and Coxeter. More precisely, we prove that  positive integral friezes are in one-to-one correspondence with ideal triangulations supplied with a collection of rescaling constants assigned to punctures. For every triangulation the set of the collections of constants is finite and is completely determined by the valencies of vertices in the triangulation. In particular, it follows that the number of non-equivalent friezes on bordered surfaces is finite, and all friezes on unpunctured surfaces are unitary. The proofs are based on Penner's decorated hyperbolic structure defined by a frieze on the surface.   
  
  \end{abstract}

 \maketitle
{\small
\setcounter{tocdepth}{1}
\tableofcontents
}

\section{Introduction}
\label{intro}

Frieze patterns (or {\em friezes} for short)  were introduced by Coxeter~\cite{C} as tables of integers of finite width satisfying the ``diamond rule''.  More precisely, a frieze of width $n$ consists of $n+2$ rows of positive integers, where the first and the last rows consist of ones, even rows are shifted with respect to the odd ones, and for every ``diamond'' of the form
{\small
  $$
  \begin{tabular}{ccccccccccccccccccccc}
&$b$\\
    $a$&&$d$\\
    &$c$\\
  \end{tabular}
  $$
}
the unimodular relation $ad-bc=1$ holds, see an example for $n=4$ below:\\

\begin{center}
{\small
\begin{tabular}{ccccccccccccccccccccc}
{\!\!$\cdots$\!\!}&&1&&1&&1&&1&&1&&1&&1&&1&&{\!\!$\cdots$\!\!}&\\
&{\!\!$\cdots$\!\!}&&1&&3&&2&&2&&1&&4&&2&&1&&{\!\!$\cdots$\!\!}\\
{\!\!$\cdots$\!\!}&&1&&2&&5&&3&&1&&3&&7&&1&&{\!\!$\cdots$\!\!}&\\
&{\!\!$\cdots$\!\!}&&1&&3&&7&&1&&2&&5&&3&&1&&{\!\!$\cdots$\!\!}\\
{\!\!$\cdots$\!\!}&&2&&1&&4&&2&&1&&3&&2&&2&&{\!\!$\cdots$\!\!}&\\
  &{\!\!$\cdots$\!\!}&&1&&1&&1&&1&&1&&1&&1&&1&&{\!\!$\cdots$\!\!}\\
\\
\end{tabular}  
}\\
\end{center}

In~\cite{CC}, Conway and Coxeter gave a complete classification of friezes in terms of triangulations of polygons: they showed that friezes of width $n$ are in one-to-one correspondence with triangulations of $(n+3)$-gons (modulo natural symmetries). This result can be reformulated in terms of Penner's $\lambda$-lengths as follows (we recall the necessary definitions in Section~\ref{prelim}). Given a triangulation $T$ of an $(n+3)$-gon, consider all triangles of $T$ as ideal triangles in the hyperbolic plane $\H^2$, with the horocycles chosen at vertices so that $\lambda$-lengths of all  edges of $T$ (including the sides of the polygon) are equal to $1$. Then the entries in the $k$-th non-trivial row of the corresponding frieze are the $\lambda$-lengths of diagonals connecting $i$-th and $(i+k+1)$-th vertices of the polygon. The diamond rule is then equivalent to the Ptolemy relation.

Connections of friezes to cluster algebras were first revealed in~\cite{Pr} and~\cite{CaCh}. In these terms, Conway--Coxeter friezes can be considered as associated to cluster algebras of type $A_n$. In particular, this gives rise to a natural generalization of the notion of frieze obtained by evaluation of cluster variables in a cluster algebra, see e.g.~\cite{ARS,MG1,FP}.  
As a consequence, friezes appear to be closely connected to various domains of mathematics, including combinatorics, representation theory, integrable systems, geometry, see~\cite{MG2} for an extensive survey.

The cluster algebras approach together with the results of~\cite{FST} provide a combinatorial construction of friezes associated to marked surfaces: in this model the entries of a frieze are assigned to arcs on a surface connecting marked points (while all the boundary arcs are assigned with $1$). Such frieze is called {\em unitary} if there exists a triangulation of the surface with vertices in marked points such that all arcs of the triangulation are assigned with $1$ (note that due to Laurent Phenomenon~\cite{FZ}, every triangulation gives rise to a unique unitary frieze). Conway--Coxeter's results assure that all friezes on a disc are unitary, which leads to a natural question:\\

{\em Are all friezes on a given surface unitary?}\\

An example of a non-unitary frieze from a once punctured disc (corresponding to cluster algebra of type $D_4$) was constructed by Thomas, see~\cite{BM}. This example was incorporated into a series of examples by Fontaine and Plamondon~\cite{FP} (see Example~\ref{ex-FP} below), who classified all friezes of type $D_n$. However, no examples of non-unitary friezes on unpunctured surfaces are known. Moreover, it was shown in~\cite{GS} that all friezes on an annulus are unitary, and in~\cite{CFGT} that all friezes on a pair of pants are also unitary.

\smallskip

In the present paper, we provide a complete classification of friezes constructed from marked bordered surfaces in the style of Conway -- Coxeter. More precisely, we prove that, similarly to the case of punctured disc, every non-unitary frieze can be obtained from a unitary frieze by multiplying the labels of all arcs incident to a given puncture by a constant. The main result can be formulated as follows.

\begin{thrm}[Theorems~\ref{all-pts},~\ref{main}, Remark~\ref{uniqueness}]
  \label{th-main}
A frieze from a marked bordered surface $S$ is uniquely defined by an ideal triangulation $T$ of $S$ and a collection of positive integers $\{k_i\}$ at all punctures $\{P_i\}$, such that $k_i$ divides the valence of $P_i$ in $T$. Every such data defines a frieze.  

  \end{thrm}
  In particular, we get the following immediate corollary (where equivalence is defined up to the action of the mapping class group, see Section~\ref{prelim}).

\setcounter{section}{4}
\setcounter{theorem}{13}
  
\begin{cor}
    All friezes on unpunctured surfaces are unitary. There is a bijective correspondence between equivalence classes of friezes and combinatorial types of ideal triangulations.
    \end{cor}

The Main Theorem also assures that every triangulation gives rise to a finite number of friezes. Since the number of combinatorial types of triangulations is finite, we get the following result.
    
\begin{cor}
  The number of equivalence classes of friezes on a given bordered surface is finite.

  \end{cor}

Our proofs are based on the decorated hyperbolic structure defined by a frieze on a surface~\cite{P}. We show that the integrality condition guarantees that the uniformization of the surface is compatible with the action of (a certain subgroup of) $SL_2(\Z)$, so the Farey triangulation on $\H^2$ induces a triangulation on the surface, and thus gives rise to a unitary frieze (where the values are $\lambda$-lengths measured with respect to Ford circles), which can be transformed to the original frieze by choosing different horocycles at punctures.

\medskip

  For closed punctured surfaces the Main Theorem does not hold. More precisely, in Section~\ref{closed} we present several examples showing that the scaling constants may not divide the valence of a puncture (Example~\ref{much}), may not be integer (Example~\ref{non-integers}), and moreover scaling of two distinct unitary friezes may lead to the same frieze (Example~\ref{non-unique}). Nevertheless, the following theorem holds for friezes on both bordered and closed surfaces.

\setcounter{section}{3}
\setcounter{theorem}{10}  

    \begin{theorem}
      Given an ideal triangulation $T$ of a marked surface $S$ and a collection of positive integers $\{k_i\}$ at all punctures $\{P_i\}$ such that $k_i$ divides the valence of $P_i$ in $T$, one can define a frieze on $S$ by scaling the corresponding unitary frieze.

      \end{theorem}

If a surface has an ideal triangulation with at least one puncture of non-prime valence, the procedure given in Theorem~\ref{all-pts} gives rise to a non-unitary frieze. 
      
        \begin{cor}
          Let $S$ be a punctured marked surface different from once punctured digon or triangle, and from twice punctured monogon. Then there exists a non-unitary frieze on $S$. 

        \end{cor}

        Also, under some additional assumptions one can guarantee that a frieze on a closed surface is unitary.

\setcounter{section}{5}
\setcounter{theorem}{0}  

    \begin{prop}
Let $F$ be a frieze on a marked surface $S$, and let $T$ be an ideal triangulation of $S$ such that the values of $F$ on the sides of every triangle of $T$ are mutually coprime. Then $F$ is unitary.

      \end{prop}

      \bigskip

      The paper is organized as follows. In Section~\ref{prelim} we remind all essential details about tagged triangulations, $\lambda$-lengths and decorated hyperbolic structures, Farey triangulation, and friezes on surfaces. Section~\ref{defs} is devoted to construction of non-unitary friezes on punctured surfaces (Theorem~\ref{all-pts}). In Section~\ref{sec_bordered}, we classify friezes on surfaces with boundary by proving that every frieze can be obtained from a unitary one by rescaling (Theorem~\ref{main}). Finally, in Section~\ref{closed} we discuss partial results and counterexamples concerning friezes on closed punctured surfaces.

      \subsection*{Acknowledgements}
The work was inspired by a question (answered in the Appendix) asked by Alain Valette at the conference ``Journ\'ees de g\'eometrie hyperbolique'' in Fribourg. We are grateful to Alain Valette for the question and to Naomi Bredon and Ruth Kellerhals for organizing the conference. The paper was written at the Max Planck Institute for Mathematics in Bonn, we thank the Institute for the financial support and excellent research environment.

      \setcounter{section}{1}
\setcounter{theorem}{0}

\section{Tagged triangulations, $\lambda$-lengths and friezes}
\label{prelim}
In this section we recall necessary results about hyperbolic surfaces and their triangulations.

\subsection{Tagged triangulations}
\label{tagged}

We briefly recall the construction of tagged triangulations of marked surfaces from~\cite{FST} (see also~\cite[Section~5]{FT}).

Let $S$ be a surface with marked points and (possibly empty) boundary, so that every boundary component contains at least one marked point. The marked points in the interior of $S$ are called {\em punctures}. We exclude a closed sphere with at most $3$ punctures, an unpunctured disc with at most $3$ marked points, and once punctured monogon.

An {\em arc} is a simple curve with endpoints in marked points defined up to isotopy. An arc is called a {\em loop} if its endpoints coincide. An {\em ideal triangulation} of $S$ is a maximal collection of mutually non-intersecting arcs. An ideal triangulation subdivides $S$ into triangles. A triangle whose two edges coincide is called a {\em self-folded triangle}.

A {\em tagged triangulation} is obtained from an ideal triangulation as follows. Let $P$ be a puncture incident to two coinciding edges of a self-folded triangle (i.e. $P$ is located inside a once punctured monogon), and let $Q$ be the other vertex of that triangle. The loop at $Q$ bounding the self-folded triangle is substituted with a copy of the arc $PQ$ tagged {\em notched at $P$}, the two obtained arcs form a {\em conjugate pair at} $P$, see Fig.~\ref{fig-self-folded}. This procedure has to be done for every self-folded triangle. Further, choose any collection of punctures (without conjugate pairs at them) and declare all arcs incident to these punctures to be tagged notched at these punctures. All the other ends of arcs are tagged {\em plain}.

In other words, tagged triangulations consist of {\em tagged} arcs, where every tagged arc is tagged {\em notched} or {\em plain} at every vertex. All tags at a boundary marked point are plain. All tags at a puncture coincide, unless there is a conjugate pair at it. 

\begin{figure}[!h]
\begin{center}
  \epsfig{file=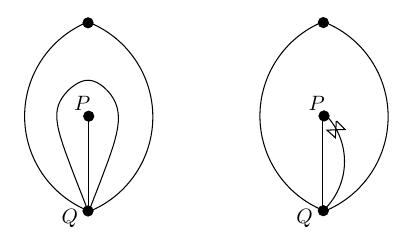,width=0.4\linewidth}
  \caption{Self-folded triangle in an ideal triangulation (left); conjugate pair at $P$ in a tagged triangulation (right).}
\label{fig-self-folded}
\end{center}
\end{figure}

Tagged triangulations undergo {\em flips}: for every arc $\gamma$ of a tagged triangulation $T$ there exists a unique tagged arc $\gamma'$ such that $T'=(T\setminus\gamma)\cup\gamma'$ is also a tagged triangulation. 

From now on, by {\em triangulation} we mean a tagged triangulation. If we need to consider an ideal triangulation we will specify this explicitly.

\subsection{Decorated hyperbolic surfaces}
\label{decorated}

Given two points $A,B\in\partial\H^2$ on the boundary of the hyperbolic plane with horocycles $h_A$ and $h_B$ centered at $A$ and $B$ respectively, Penner~\cite{P} defined {\em $\lambda$-length} $\lambda_{AB}$ as $\exp(d/2)$, where $d$ is the signed distance between the horocycles $h_A$ and $h_B$. Here $d<0$ if the horocycles intersect in two points (in which case $d$ is the negative of the distance between $AB\cap h_A$ and $AB\cap h_B$).

In the upper halfplane model of $\H^2$, horocycles are Euclidean circles tangent to the real line (or horizontal lines for horocycles at $\infty$). An easy computation shows that the $\lambda$-lengths depend on the choice of horocycles as follows: dividing the Euclidean diameter of a horocycle by $k^2$ (or multiplying the $y$-coordinate of a horizonal line by $k^2$) multiplies the corresponding $\lambda$-length by $k$.

As it was shown in~\cite{P}, $\lambda$-lengths in $\H^2$ satisfy {\em Ptolemy relation}. Namely, given an ideal quadrilateral $ABCD$ with a horocycle at every vertex, one has
\begin{equation}
  \label{eq_Ptolemy}
  \lambda_{AC}\lambda_{BD}=\lambda_{AB}\lambda_{CD}+\lambda_{AD}\lambda_{BC}.
  \end{equation}

Given a set $M$ of marked points on $S$, a {\em decorated hyperbolic structure} on $(S,M)$ is a complete hyperbolic metric on $S\setminus M$ together with a chosen horocycle at every point of $M$. Penner~\cite{P} shows that  given an ideal triangulation $T$ of a marked surface $(S,M)$ with positive numbers assigned to each arc of $T$, there exists a unique decorated hyperbolic structure on $(S,M)$ such that the assigned numbers are precisely $\lambda$-lengths of the corresponding arcs of $T$. 

Results of~\cite{P} (and the Uniformization Theorem) imply that, given an ideal triangulation $T$ with arcs labeled by positive numbers, $S$ can be represented as a quotient of a triangulated domain $\Omega$ in the hyperbolic plane $\H^2$ by a certain discrete group $\Gamma$ of isometries of $\H^2$, such that the following hold:
\begin{itemize}
\item[-]
  $\Omega$ is an infinite polygon bounded by preimages of boundary arcs of $S$;
\item[-]
  collection of horocycles at preimages of every puncture is $\Gamma$-invariant;
\item[-]
  the $\lambda$-lengths of the arcs of the triangulation of $\Omega$ are precisely the labels of their images in $S$.
  \end{itemize}
From now on, we will write $S$ instead of $(S,M)$ assuming there is no ambiguity.

\subsection{$\lambda$-lengths of tagged arcs}
\label{lambda}
In this section we recall from~\cite{FT} the hyperbolic geometry related to tagged triangulations. 

Let $S$ be a marked surface with decorated hyperbolic structure, let $P$ be a puncture and $h$ be the corresponding horocycle. One can uniquely define a {\em conjugate horocycle} $\bar h$ at $P$ as follows: if the length of $h$ as a (non-geodesic) curve in the hyperbolic metric is $l(h)$, then $\bar h$ is defined by $l(\bar h)=1/l(h)$. The $\lambda$-length of any arc tagged notched at $P$ is then defined as follows: take the same arc tagged plain at $P$, and compute the $\lambda$-length with respect to the horocycle $\bar h$ at $P$.   

The $\lambda$-lengths of homotopic arcs $\gamma$ and $\bar\gamma$ from $Q$ to $P$ tagged differently at $P$ are related as follows:
\begin{equation}
  \label{eq_loop}
  \lambda_\gamma\lambda_{\bar\gamma}=\lambda_l,
  \end{equation}
where $l$ is the loop at $Q$ going around $P$ by following $\gamma$ as in Fig.~\ref{fig-flip}, left.

If an arc $\gamma$ in a tagged triangulation is not a part of a conjugate pair, and neither are all adjacent arcs, then the $\lambda$-lengths of $\gamma$ and the flipped arc $\gamma'$ are related by the Ptolemy relation. In general, the relation between $\gamma$ and $\gamma'$ needs to be adjusted, see~\cite[Section~8]{FT}. In particular, flipping an internal arc of a once punctured digon results in the following {\em digonal relation} (see Fig.~\ref{fig-flip}, right) for $\theta'=\bar\gamma$:
\begin{equation}
  \label{eq_digonal}
  \lambda_\theta\lambda_{\theta'}=\lambda_\alpha+\lambda_\beta.
  \end{equation}
We will follow~\cite{FT} by calling various relations on $\lambda$-lengths of tagged arcs by (generalized) {\em Ptolemy relations}.

\begin{figure}[!h]
\begin{center}
  \epsfig{file=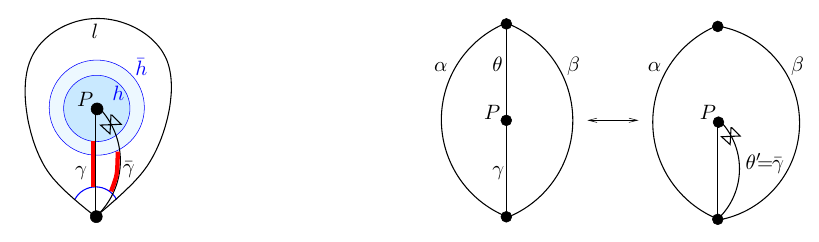,width=0.9\linewidth}
  \caption{$\lambda$-lengths of tagged arcs. Left: conjugate horocycle; right: flip of an arc in a punctured digon.}
\label{fig-flip}
\end{center}
\end{figure}

\subsection{Farey triangulation and Ford circles}
\label{Farey}

Recall that the {\em Farey graph} on $\H^2$ (in the upper halfplane model) has vertices in all rational points of $\R$ (and $\infty=\frac10$), with two vertices $\frac{p}q$ and $\frac{u}v$ being joined by an edge (represented by a hyperbolic geodesic) if $|pv-qu|=1$ (we assume all fractions to be reduced). A {\em Ford circle} at a point $\frac{p}q$ is a horocycle centered at $\frac{p}q$ of Euclidean diameter $\frac1{q^2}$ (or a horizontal Euclidean line $\mathrm{Im}\, z=1$ as a horocycle at $\infty$). Farey graph provides a triangulation of $\H^2$ with every triangle being an ideal triangle, with $\lambda$-lengths of all edges with respect to Ford circles being equal to $1$. The set of Ford circles is invariant with respect to the action of $PSL_2(\Z)\subset \mathrm{Isom}(\H^2)$. 

The $\lambda$-lengths between arbitrary two vertices of the Farey graph with respect to Ford circles are expressed as follows (see e.g.~\cite{MGOT,FKST}): if the reduced fractions for the two points are  $\frac{p}q$ and $\frac{u}v$, then
\begin{equation}
  \label{eq_det}
  \lambda_{\frac{p}q,\frac{u}v}=|pv-qu|.
  \end{equation}

\subsection{Friezes constructed from surfaces}
\label{friezes}

Recall from previous sections that given a marked surface $S$ and a triangulation $T$, an assignment of $\lambda$-lengths to all arcs of $T$ defines uniquely a decorated hyperbolic structure on $S$ and thus $\lambda$-lengths of all tagged arcs. 

Given a marked surface $S$, a {\em positive frieze} on $S$ is an assignment of $\lambda$-lengths to all arcs of some triangulation $T$, or, equivalently, a decorated hyperbolic structure (and thus a map $F:\{\gamma\}\to\R$ defined on the set of tagged arcs, where $F(\gamma)$ is the $\lambda$-length of $\gamma$). We require $\lambda$-lengths of all boundary arcs to be equal to $1$. 

We note that usually friezes are defined as ring homomorphisms of the corresponding cluster algebra $\mathcal{A}(S)$ to $\R$ (see e.g.~\cite{ARS,MG1}), where a frieze is {\em positive} if the range of the map is in $\R_+$. Due to results of~\cite{P,FST,FT} the definition we use is equivalent to the usual one. Here, the cluster variables are represented by tagged arcs, and the exchange relations for cluster variables are transformed into Ptolemy relations on the corresponding $\lambda$-lengths.

A positive frieze $F$ is {\em integral} if $F(\gamma)\in\Z_+$ for every tagged arc $\gamma$ on $S$. From now on, we will consider positive integral friezes only, so we call them {\em friezes} for short.

A frieze is {\em unitary} if there exists a triangulation $T$ of $S$ such that $F(\gamma)=1$ for every $\gamma\in T$ (by Ptolemy relation, such triangulation, if exists, is unique). Equivalently, all $\lambda$-lengths of arcs of $T$ are equal to $1$, i.e. the two horocycles at the endpoints of any arc of $T$ are tangent.   

For example, it follows from~\cite{CC} that all friezes on an unpunctured disc (i.e. a polygon) are unitary.

\section{Non-unitary friezes on punctured surfaces}
\label{defs}

In this section, we provide a construction of non-unitary friezes on punctured surfaces. We start with an example from~\cite{FP}.

\begin{example}
\label{ex-FP}
In~\cite{FP}, Fontaine and Plamondon construct a series of non-unitary friezes on a once punctured disc $D_n$ with $n$ boundary marked points  as follows. Take any tagged triangulation $T$ of $D_n$, define $F(\gamma)=1$ for all arcs not incident to the puncture, and $F(\gamma)=k$ for all arcs incident to the puncture, where $k$ divides the valence of the puncture in $T$ (here the valence is the number of ends of arcs at the puncture in the corresponding ideal triangulation). It is also shown in~\cite{FP} that all non-unitary friezes on $D_n$ can be constructed in this way.

  \end{example}

  \begin{remark}
    \label{disc}
    Let us make the following observation which will be the key for our further considerations. The non-unitary friezes described in Example~\ref{ex-FP} can be obtained from unitary friezes by changing the horocycle at the puncture (more precisely, the Euclidean size of the horocycle in the upper halfplane model of $\H^2$ is divided by $k^2$).

    Note that the condition on the divisibility is implied by the equation~(\ref{eq_loop}) which states that $\lambda_\gamma\lambda_{\bar\gamma}=\lambda_l$ for any conjugate pair $(\gamma,\bar\gamma)$, where $l$ is the loop following $\gamma$ as in Fig.~\ref{fig-flip}, left. Namely, given any tagged triangulation $T$ with all $\lambda$-lengths equal to $1$, and an arc $\gamma\in T$ incident to the puncture, a direct computation shows that the value of $\lambda_l$ is equal to the valence of the puncture in the corresponding ideal triangulation (and it stays intact under change of the horocycle at the puncture). Thus, equation~(\ref{eq_loop}) implies that for any $\gamma\in T$ incident to the puncture, $\lambda_\gamma$ has to divide the valence after any change of the horocycle.    
    \end{remark}

The remaining part of the section is devoted to applying the observations made in Remark~\ref{disc} to all punctured surfaces.

\begin{definition}
  \label{epsilon}
  Let $S$ be a marked decorated hyperbolic surface, denote by $\{P_i\}$ the marked points on $S$. Let $\gamma$ be an arc on $S$. Define an integer $\varepsilon_i(\gamma)$ as follows:
  \begin{itemize}
  \item[-]
    if $\gamma$ is tagged plain at $P_i$, then $\varepsilon_i(\gamma)$ is the number of ends of $\gamma$ at $P_i$;
\item[-]
  if $\gamma$ is tagged notched at $P_i$, then $\varepsilon_i(\gamma)$ is the negative of the number of ends of $\gamma$ at $P_i$.
\end{itemize}
\end{definition}

\begin{definition}
  \label{valence}
  Given a triangulation $T$ of $S$ and a puncture $P_i$, the {\em valence} $\val_T(P_i)$ of $P_i$ in $T$ is defined as the number of ends of arcs at $P_i$ in the corresponding ideal triangulation.  
\end{definition}

\begin{definition}
Given a unitary frieze $F$ on $S$, a triangulation $T$ of $S$ is {\em unitary} if $F(\gamma)=1$ for every arc $\gamma$ of $T$. 
  \end{definition}

  \begin{remark}
    \label{un_ideal}
Let $T$ be a triangulation, consider the corresponding ideal triangulation $\widetilde T$. Then equation~(\ref{eq_loop}) implies that $T$ is unitary if and only if $\lambda_\gamma=1$ for any $\gamma\in \widetilde T$.  
    \end{remark}
  
  \begin{remark}
    \label{unitary}
    Let $T$ be a unitary triangulation of $S$. Take any triangle of the corresponding ideal triangulation $\widetilde T$ and lift it to $\H^2$ to a triangle with vertices $0,1,\infty$ and with Ford circles as horocycles (which is possible by Remark~\ref{un_ideal}). Gluing the adjacent triangles of $\widetilde T$, we obtain adjacent triangles of the Farey triangulation. Therefore, $S$ can be thought as a quotient of a domain $\Omega\subset\H^2$ by a discrete group $\Gamma\subset\mathrm{Isom}(\H^2)$, such that preimages of every marked point of $S$ are rational numbers, and the preimages of the horocycles in the corresponding decorated hyperbolic structure are precisely Ford circles.   
    \end{remark}

  \begin{lemma}
    \label{loop}
    Let $S$ be a marked decorated hyperbolic surface, and suppose there is a unitary triangulation $T$ of $S$. Let $P$ be a puncture, and let $PQ$ be an arc of $T$ ($P$ and $Q$ are distinct). Then the $\lambda$-length of the loop at $Q$ going along the arc $PQ$ around $P$ is equal to the absolute value of the valence of $P$ in $T$.  
    \end{lemma}

    \begin{proof}
Represent $S$ as a quotient of a domain in the hyperbolic plane $\H^2$, such that $PQ$ is lifted to the line  from $\tilde P=\infty$ to $\tilde Q=0$ with Ford circles as horocycles. Since $T$ is unitary, the arcs of the ideal triangulation $\widetilde T$ that are incident to $P$ are lifted to the lines connecting $\infty$ with consecutive integers, also with Ford circles as horocycles at the endpoints. Then the $\lambda$-length of the loop is precisely the $\lambda$-length between $\tilde Q=0$ and the next lift $\tilde Q^{(1)}$ of $Q$, i.e. $m$, where $m$ is the valence of $P$. By~(\ref{eq_det}), $\lambda_{\tilde Q\tilde Q^{(1)}}=m$.  
      
      \end{proof}

      \begin{remark}
Note that while flipping the arcs of a unitary triangulation incident to $P$, the sum of $\lambda$-lengths of ``third sides'' of the triangles incident to $P$ (in the corresponding ideal triangulations) remains intact, and thus it is equal to the valence of $P$ in the unitary triangulation.  
        \end{remark}

      \begin{remark}
        \label{notched}
        Assume that all arcs in a unitary triangulation $T$ are tagged plain (except for conjugate pairs). It follows from Lemma~\ref{loop} and~(\ref{eq_loop}) that the $\lambda$-length of the arc $PQ$ notched at $P$ is equal to the valence of $P$ in $T$ (denote it by $m$). Furthermore, to compute the $\lambda$-length of an arc incident to $P$ with changed tag at $P$ one needs to use the conjugate horocycle at $P$ instead of the original one. This implies that the $\lambda$-length of {\em every} arc tagged notched at $P$ is divisible by $m$ (and $\lambda$-lengths of loops at $P$ are divisible by $m^2$). Moreover, by the same reason, if an arc in $S$ tagged plain on both ends connects two punctures $P$ and $P'$, then the same arc tagged  notched on both sides has $\lambda$-length divisible by the product of valencies of $P$ and $P'$ in $T$.    
        
        \end{remark}

    \begin{theorem}
\label{all-pts}
Let $\hat F$ be a unitary frieze on a surface $S$, and let $T$ be the corresponding unitary triangulation (we assume that all arcs of $T$ except for conjugate pairs are tagged plain). Let $\{P_i\}$ be all marked points of $S$, denote $m_i=\val_T(P_i)$ (with $m_i=1$ if $P_i\in\partial S$), and let $\{k_i\}$ be divisors of $\{m_i\}$. Then there exists a frieze $F$ on $S$ such that  for any arc $\gamma$ on $S$ one has $\displaystyle{F(\gamma)=\prod_i\!k_i^{\varepsilon_i(\gamma)}\hat F(\gamma)}$, where $\varepsilon_i(\gamma)$ is as in Definition~\ref{epsilon}.

      \end{theorem}

      \begin{proof}
        Consider the frieze $\hat F$, and define the frieze $F$ as follows: we change the horocycles at punctures $P_i$ by dividing their Euclidean size by $k_i^2$, so all $\lambda$-lengths of arcs of $T$ change as required. We are left to show that $\lambda$-lengths of all other arcs remain integer.

        Consider any arc $\gamma$ in $S$ with endpoints in $P_i$ and $P_j$, its $\lambda$-length is not affected by changes of all horocycles except for the ones at $P_i$ and $P_j$.

        After change of the horocycle at $P_i$ only, the $\lambda$-lengths of all arcs tagged plain at $P_i$ are multiplied by $k_i$ (or $k_i^2$ for loops), and the $\lambda$-lengths of all notched arcs at $P_i$ are divided by $k_i$ (or $k_i^2$ for loops). Similarly, the same holds for the change of the horocycle at $P_j$. Now, all $\lambda$-lengths remain integer according to Remark~\ref{notched}.

        \end{proof}

        \begin{cor}
          \label{non-unitary}
          Let $S$ be a punctured marked surface different from once punctured digon or triangle, and from twice punctured monogon. Then there exists a non-unitary frieze on $S$. 

        \end{cor}
\begin{proof}
        According to Theorem~\ref{all-pts}, it is sufficient to find a triangulation with at least one puncture having a non-prime valence. It is easy to see that all punctured surfaces except for the ones mentioned in the statement admit such triangulations.

\end{proof}
        
\section{Friezes on bordered surfaces}
\label{sec_bordered}

In this section, we classify friezes on surfaces with non-empty boundary.

The main tool in our considerations is the following technical result.

\begin{prop}
  \label{divisor_boundary}
  Let $F$ be a frieze on a bordered surface $S$, let $P$ be a marked point, let $AB$ be a boundary arc (any of $A$, $B$ and $P$ may coincide). Suppose that there exists prime $p$ and arcs $\alpha$ and $\beta$ both tagged plain (resp., notched) with endpoints $P,A$ and $P,B$ respectively, such that $\a,\b$ and $AB$ form a triangle on $S$, and both $F(\a)$ and $F(\b)$ are divisible by $p$. Then  for every arc $\gamma$ tagged plain (resp., notched) at $P$ with the other endpoint at $A$ or $B$ the value $F(\gamma)$ is divisible by $p$.
  \end{prop}

We split the proof of  Proposition~\ref{divisor_boundary} into Lemmas~\ref{0}--\ref{many}.

  \begin{lemma}
    \label{0}
In the assumptions of Proposition~\ref{divisor_boundary}, let $\eta$ be any arc incident to $P$ and tagged at $P$ in the same way as $\a$ and $\b$. Suppose that $\eta$ intersects neither $\a$ nor $\b$. Then $F(\eta)$ is divisible by $p$.
    
    \end{lemma}
  
    \begin{proof}
      Take an arbitrary arc $\eta$ on $S$ tagged plain (resp. notched) at $P$, $\eta$ does not intersect the triangle $ABP$. Denote the other endpoint of $\eta$ by $Q$ ($Q$ and $P$ may coincide). Consider a quadrilateral with sides $\eta$, $\a$ and $AB$ and diagonal $\beta$ (it can be constructed as follows: the forth side $BQ$ follows $\b$ and then $\eta$, where $BQ$ is tagged at $Q$ in the same way as $\eta$, see Fig.~\ref{fig-0}, left; note that if $P$ coincides with $A$ or $B$ one may need to switch $\a$ and $\b$ in the construction). If $\lambda_{\a}$ and $\lambda_{\b}$ are divisible by $p$, then by the Ptolemy relation~(\ref{eq_Ptolemy}) $\lambda_{\eta}\lambda_{AB}$ is also divisible by $p$, which implies the lemma as  $\lambda_{AB}=1$.

      \end{proof}

\begin{figure}[!h]
\begin{center}
 \epsfig{file=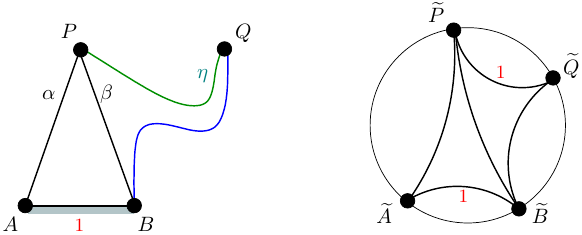,width=0.63\linewidth}
  \caption{To the proof of Lemmas~\ref{0} (left) and~\ref{nok} (right). 
  }
\label{fig-0}
\end{center}
\end{figure}

\begin{lemma}
  \label{nok}
In the assumptions of Proposition~\ref{divisor_boundary}, $P$ must be a puncture. In particular, $P\ne A,B$.  

  \end{lemma}

  \begin{proof}
    Suppose that $P$ is  a boundary marked point. If $P$ is distinct from $A$ and $B$, then Lemma~\ref{0} implies that the $\lambda$-length of any boundary arc incident to $P$ is divisible by $p$, which leads to a contradiction. So, we may assume that $P$ coincides with $A$ or $B$, say $P=A$.

    Consider $S$ as a quotient of a domain in the hyperbolic plane $\H^2$, choose one preimage of the triangle $ABP$ and denote its vertices by $\tilde A,\tilde B,\tilde P$. Since $A=P$, there is a preimage of the boundary arc $AB$ incident to $\tilde P$, denote its other endpoint by $\tilde Q$. One of the geodesics $\tilde Q\tilde A$ and $\tilde Q\tilde B$ (say, $\tilde Q\tilde B$) does not intersect the triangle $\tilde A\tilde B\tilde P$  (see Fig.~\ref{fig-0}, right), so there is an ideal quadrilateral in $\H^2$ with opposite sides $\tilde A\tilde B$ and $\tilde P\tilde Q$ of unit $\lambda$-lengths.

    Observe that the images of sides $\tilde A\tilde Q$ and $\tilde B\tilde Q$ on $S$ may self-intersect. However, using the Ptolemy relation, it is easy to see that $\lambda$-lengths of self-intersecting arcs can be written as sums and products of $\lambda$-lengths of non-self-intersecting arcs and simple closed curves (see~\cite{MW,FT}). In view of~\cite[Section 8]{MSW} and~\cite[Section 5]{CLS}, the $\lambda$-lengths of simple closed curves can be written as sums and products of $\lambda$-lengths of non-self-intersecting arcs. Therefore, the $\lambda$-lengths of all non-self-intersecting arcs are integers.

    We can now apply the Ptolemy relation to the quadrilateral with vertices $\tilde A,\tilde B,\tilde P,\tilde Q$. The opposite sides $\tilde A\tilde B$ and $\tilde P\tilde Q$ have unit $\lambda$-lengths, and the $\lambda$-lengths of the side  $\tilde A\tilde P$ and the diagonal  $\tilde B\tilde P$ are divisible by $p>1$, so we get a contradiction.

    \end{proof}

  The further proof of Proposition~\ref{divisor_boundary} is by induction on the number of intersections of $\gamma$ with the sides of the triangle $ABP$. Lemma~\ref{0} serves as the base of the induction.

    \begin{lemma}
      \label{1}
In the assumptions of Proposition~\ref{divisor_boundary}, let $\gamma$ intersect one of $\a$ or $\b$ precisely once (and not intersect the other). Then $F(\gamma)$ is divisible by $p$.      

\end{lemma}

\begin{proof}
  We need to consider two cases: either $A$ and $B$ coincide or not. Denote the triangle formed by $\a,\b$ and $AB$ by $\Delta$.

  Suppose first that $A$ and $B$ are distinct, assume $\gamma$ intersects $\b$ (and thus $\gamma$ is incident to $A$, see Fig.~\ref{fig-1}, left).  Consider a boundary arc $BC$, and define two new arcs $BP$ and $CP$ as follows. The arc $BP$ goes from $B$ along $\b$ till the intersection with $\gamma$, and then follows $\gamma$ till $P$. The arc $CP$ follows $CB$ and then $BP$. Then the arcs $BP$ and $CP$ do not intersect the triangle $\Delta$, and thus, by Lemma~\ref{0}, their $\lambda$-lengths are divisible by $p$. Moreover, $BP$, $CP$ and the boundary arc $BC$ form a triangle $\Delta'$, and $\gamma$ does not intersect $\Delta'$. Applying Lemma~\ref{0} again, we see that $\lambda_\gamma$ is divisible by $p$.

  Consider now the case $A=B$, assume again that $\gamma$ intersects $\b$. We introduce two additional arcs $\delta$ and $\psi$ as follows (see Fig.~\ref{fig-1}, right). The arc $\delta$ starts at $P$ and goes along $\b$ till the intersection with $\gamma$, and then follows $\gamma$ back to $P$. The arc $\psi$ starts at $A$ and goes along $\a$ and $\gamma$ avoiding an intersection with $\delta$, see Fig.~\ref{fig-1}, right. Now the loops $\delta$ and $\psi$ bound an annulus consisting of two triangles: one has sides $\a,\gamma,\delta$, and the other with sides $\a, \gamma,\psi$. Making a flip in $\a$, we obtain another arc $\a'$ with $$\lambda_{\a}\lambda_{\a'}=\lambda_\psi\lambda_\delta+\lambda_\gamma^2.$$
  Observe that $\lambda_\delta$ is divisible by $p$ by Lemma~\ref{0}, and  $\lambda_{\a}$ is divisible by $p$ by the assumption of Prop.~\ref{divisor_boundary}. Therefore, $\lambda_\gamma^2$ is divisible be $p$, and thus $\lambda_\gamma$ is divisible by $p$ since $p$ is prime.   

      \end{proof}

\begin{figure}[!h]
\begin{center}
 \epsfig{file=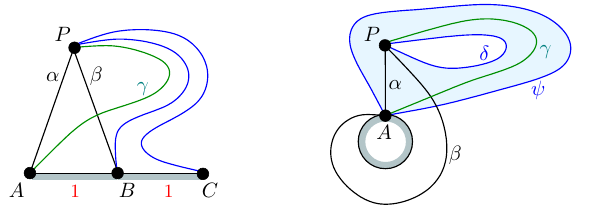,width=0.8\linewidth}
  \caption{To the proof of Lemma~\ref{1}: $\gamma$ intersects $\a$ or $\b$ once. Left: points $A$ and $B$ are distinct; right: $A$ coincides with $B$.
  }
  
\label{fig-1}
\end{center}
\end{figure}

The following lemma completes the proof of Proposition~\ref{divisor_boundary}.      

   \begin{lemma}
      \label{many}
In the assumptions of Proposition~\ref{divisor_boundary}, let the total number of intersections of $\gamma$ with $\a$ and $\b$ be bigger than one. Then $F(\gamma)$ is divisible by $p$.      

\end{lemma}

\begin{proof}
  Denote the total number of intersection points of $\gamma$ with $\a$ and $\b$ by $n$. Our aim is to construct a triangle with the same side $AB$ and vertex $P$ such that each of the sides $AP$ and $BP$ intersect $\a$ and $\b$ in less than $n$ points, and the total number of intersections of $\gamma$ with $AP$ and $BP$ is also less than $n$. By the induction assumption, this will imply that $\lambda_{AP}$ and $\lambda_{BP}$ are divisible by $p$, and thus we can apply the induction assumption to the new triangle to conclude that $\lambda_\gamma$ is also divisible by $p$.

  Denote the triangle formed by $\a,\b$ and $AB$ by $\Delta$. Orient $\gamma$ towards $P$, and consider the set $\Sigma$ of all intersection points of $\gamma$ with $\a$ and $\b$ in which $\gamma$ enters the triangle $\Delta$. Since $\gamma$ is not self-intersecting, either there exists $M\in\Sigma$ such that the segment of the side of $\Delta$ between $M$ and $AB$ does not contain any intersection points with $\gamma$ (see Fig.~\ref{fig-many}, left and right), or, otherwise, there exists $M\in \Sigma$ such that the segment of the side of $\Delta$ between $M$ and $AB$ contains a unique intersection point with $\gamma$, and it does not belong to $\Sigma$ (see Fig.~\ref{fig-many}, middle). Note that in the latter case there might be two such points, we choose the one lying on the segment of $\gamma\cap\Delta$ that is closer to $AB$.

  By symmetry, we can assume that in the former case $M\in\a$, and in the latter case $\gamma$ is incident to $A$. We can now define the sides of new triangle $\Delta'$. In the former case its side $AP$ goes from $A$ along $\a$ till $M$ and then goes along $\gamma$ till $P$, and $BP$ goes along $BA$ and then $AP$ (see Fig.~\ref{fig-many}, left and right). In the latter case $BP$ goes from $B$ along $\b$ till $M$ and then goes along $\gamma$ till $P$, and $AP$ goes along $AB$ and then $BP$ (see Fig.~\ref{fig-many}, middle).

  The total number of intersections of $AP$ with $\a$ and $\b$ is at most $n-1$, and the same holds for $BP$, so their $\lambda$-lengths are divisible by $p$. The arc $\gamma$ intersects the sides of the new triangle $\Delta'$ at most once, and thus its $\lambda$-length is also divisible by $p$ by Lemmas~\ref{0} and~\ref{1}.

      \end{proof}

\begin{figure}[!h]
\begin{center}
  \epsfig{file=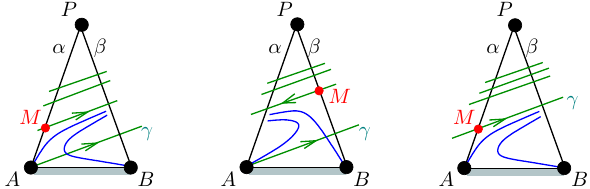,width=0.89\linewidth}
  \caption{To the proof of Lemma~\ref{many}: reducing the number of intersections.
  }
\label{fig-many}
\end{center}
\end{figure} 

\begin{remark}
  \label{prime}
The only place in the proof of Prop.~\ref{divisor_boundary} where we used that $p$ is prime is the proof of Lemma~\ref{1} (and in the case of $A=B$ only). In particular, Lemmas~\ref{0} and~\ref{nok} hold for arbitrary integer $k>1$. Also, Proposition~\ref{divisor_boundary} holds for every integer $k>1$ if $A\ne B$. 
  \end{remark}

      With  Proposition~\ref{divisor_boundary} in hands, we can now prove a stronger statement.

\begin{prop}
  \label{divisor_all}
  In the assumptions of Proposition~\ref{divisor_boundary}, for every arc $\gamma$ tagged plain (resp., notched) at $P$ the value $F(\gamma)$ is divisible by $p$ (and by $p^2$ for loops at $P$).  
  \end{prop}

  \begin{proof}
Take an arbitrary arc $PQ$ on $S$ tagged plain (resp. notched) at $P$, consider a quadrilateral with opposite sides $PQ$ and $AB$ (it can be constructed as follows: take any arc with endpoints $B,Q$, and then the fourth arc $AP$ follows $AB$, $BQ$ and then $QP$, where $AP$ is tagged at $P$ in the same way as $PQ$). Since $\lambda_{AP}$ and $\lambda_{BP}$ are divisible by $p$ (by Prop.~\ref{divisor_boundary}), $\lambda_{PQ}\lambda_{AB}$ is also divisible by $p$ by the Ptolemy relation, which implies the proposition as  $\lambda_{AB}=1$. If $Q=P$, then all diagonals and two opposite sides of the quadrilateral are divisible by $p$, so  $\lambda_{PQ}\lambda_{AB}=\lambda_{PQ}$ is divisible by $p^2$.
    
  \end{proof}

Proposition~\ref{divisor_all} allows us to define a new frieze on $S$ once the assumptions of Proposition~\ref{divisor_boundary} are satisfied.

  \begin{cor}
    \label{rescaling}
In the assumption of Proposition~\ref{divisor_boundary}, for any arc $\gamma$ on $S$ define  $$F'(\gamma)=p^{-\varepsilon(\gamma)}F(\gamma),$$ where $\varepsilon$ is as in Definition~\ref{epsilon} applied to puncture $P$. Then $F'$ is a frieze, i.e. $F'(\gamma)\in\Z$ for any arc $\gamma$.
    
\end{cor}

\begin{remark}
  \label{rescale_h}
Frieze $F'$ corresponds to the same hyperbolic metric on $S$ as $F$ does, with a slightly different decoration: the horocycle at $P_i$ is amended. More precisely, if we consider $S$ as a quotient of a domain in the hyperbolic plane $\H^2$, the Euclidean size of the horocycle at every preimage of $P_i$ is multiplied by $p^2$ compared to $F$ (cf. Remark~\ref{disc}).  
  
  \end{remark}

The next lemma is the final preparatory result needed for the proof of Theorem~\ref{main}.

 \begin{lemma}
  \label{divisor}
  Let $T$ be a unitary triangulation of a bordered surface $S$ with marked points $\{P_i\}$ of valencies $\{m_i\}$ (where $m_i=1$ for $P_i\in\partial S$), and assume that all arcs in $T$ are tagged plain (except for conjugate pairs). Let $\{k_i\}$ be a collection of positive integers (with $k_i=1$ for all boundary marked points), and suppose that there is a positive integral frieze $F$ on $S$ defined by $\displaystyle{F(\gamma)=\prod_i\!k_i^{\varepsilon_i(\gamma)}\lambda_\gamma}$. Then $k_i$ divides $m_i$ for every $i$.
  \end{lemma}

  \begin{proof}
    Since $k_j=1$ for boundary marked points, for $F$ being a frieze it is necessary that every arc notched at $P_i$ with other endpoint at the boundary has $\lambda$-length divisible by $k_i$ (see Remark~\ref{notched}). By Prop.~\ref{divisor_all}, this implies that every arc tagged notched at $P_i$ has  $\lambda$-length divisible by $k_i$ (as $F(AB)=1$ for every boundary arc $AB$). Now, take any arc of $T$ with endpoint $P_i$. As it was observed in Remark~\ref{notched}, the $\lambda$-length of the corresponding arc tagged notched at $P_i$ is a divisor of $m_i$, so $k_i$ should also divide $m_i$.  
    
  \end{proof}

\begin{theorem}
  \label{main}
  Let $S$ be a marked surface with boundary and marked points $\{P_i\}$, and let $F$ be a frieze on $S$. Then there exists a unitary frieze $\hat F$ on $S$ (with a unitary triangulation $T$) and a collection of positive integers $\{k_i\}$ (where $k_i=1$ for $P_i\in\partial S$), such that $k_i$ divides $\val_T(P_i)$ and $\displaystyle{F(\gamma)=\prod_i\!k_i^{\varepsilon_i(\gamma)}\hat F(\gamma)}$  for every arc $\gamma$ in $S$.

\end{theorem}

  \begin{proof}

    Let $F$ be a frieze on $S$. Recall from Section~\ref{prelim} that $S$ can be represented as a quotient of a domain $\Omega\subset\H^2$ by a certain discrete group. The domain $\Omega$ is an infinite polygon bounded by lifts of boundary arcs of $S$, and $\Omega$ is defined uniquely up to isometry of $\H^2$.

    Choose one boundary segment $\delta=P_0P_\infty$ of $S$ (the two marked points may coincide), and  lift the segment to $\H^2$ to the line from $\tilde P_0=0$ to $\tilde P_\infty=\infty$. Since $\lambda_{P_0P_\infty}=1$, the corresponding horocycles are tangent, place the common point of two horocycles at $i$ (i.e. the horocycles at $\tilde P_0$ to $\tilde P_\infty$ are Ford circles). This choice then defines a lift of any arc $P_0P_i$ or $P_\infty P_i$ uniquely. By lifting every triangle  $P_0P_\infty P_i$ in $S$ with side $P_0P_\infty=\delta$, we obtain all preimages of marked points.

Consider any triangle $P_0P_\infty P_i$ on $S$ with side $\delta$. If $F(P_0P_i)$ and $F(P_\infty P_i)$ are not coprime, then we can use Corollary~\ref{rescaling} to define a new frieze $F'$, such that $F(\gamma)=p^{\varepsilon_i(\gamma)}F'(\gamma)$ for certain prime $p$. Applying this procedure iteratively, we may assume that $F(\gamma)=\prod_ik_i^{\varepsilon_i(\gamma)}\hat F(\gamma)$ for some collection of integers $\{k_i\}$ and some frieze $\hat F$, where  $\hat F(P_0P_i)$ and $\hat F(P_\infty P_i)$ are coprime for any triangle $P_0P_\infty P_i$ on $S$.  We are left to prove that $\hat F$ is unitary. 

Let $01\tilde P_i$ be the lift of the triangle $P_0P_\infty P_i$ considered above. An elementary computation shows that $\tilde P_i=\pm \hat F(P_0P_i)/\hat F(P_\infty P_i)$, and the horocycle of $\hat F$ at $\tilde P_i$ has Euclidean diameter $\frac1{\hat F(P_\infty P_i)^2}$, i.e. it is a Ford circle.


In particular, all $\tilde P_i$ are rational, i.e. they are vertices of the Farey graph. Moreover, the horocycle at every preimage $\tilde P_i$ is a Ford circle. In other words, Ford circles at all  preimages of $P_i$ are mapped to the same horocycle.


Now, the Farey graph provides a triangulation of $\Omega$ with all edges of $\lambda$-length $1$ (with respect to Ford circles). This triangulation induces an ideal triangulation $\widetilde T$ of $S$, all arcs of $\widetilde T$ have $\lambda$-length $1$ with respect to the images of the Ford circles (note that $\widetilde T$ is indeed an ideal triangulation: $\lambda$-lengths of all geodesic arcs are integers as they come from the Farey graph, so any geodesic of $\lambda$-length $1$ is non-self-intersecting, and any two do not intersect -- this can be seen by applying the Ptolemy relation and using integrality of all $\lambda$-lengths as in~\cite[Lemma 3.3]{FP}). Therefore, in view of Remark~\ref{un_ideal} the frieze $\hat F$ is unitary, and  we have   $F(\gamma)=\prod_ik_i^{\varepsilon_i(\gamma)}\hat F(\gamma)$ as required. Finally, every $k_i$ is a divisor of the $\val_T(P_i)$ as shown in Lemma~\ref{divisor} (where $T$ is any tagged triangulation with ideal triangulation $\widetilde T$), which completes the proof.

  \end{proof}

  \begin{remark}
    \label{uniqueness}
It is easy to see from the proof that the ideal triangulation $\widetilde T$ in Theorem~\ref{main} is unique: by sending a boundary arc to the line from $0$ to $\infty$ with Ford circles, all arcs of any unitary triangulation are lifted to arcs of the Farey graph, and the image of the Farey graph on $S$ is uniquely defined. Therefore, unitary triangulation $T$ is unique up to change of taggings of punctures. The numbers $k_i$ can be found as $k_i=\gcd(F(P_0P_i),F(P_\infty P_i))$ and do not depend on the choice of arc $P_0P_i$.  Change of tagging of $T$ at $P_i$ corresponds to taking $k_i=\val_T(P_i)$.     
    
    \end{remark}

\begin{definition}  
  Friezes on $S$ are {\em equivalent} if there exists an element of the mapping class group of $S$ taking one frieze to the other.
  
\end{definition}

  We get some immediate corollaries of Theorem~\ref{main}. 
  
\begin{cor}
    \label{unpunctured}
    All friezes on unpunctured surfaces are unitary. There is a bijective correspondence between equivalence classes of friezes and combinatorial types of ideal triangulations.
    \end{cor}

\begin{cor}
\label{finite}
  The number of equivalence classes of friezes on a given bordered surface is finite.

  \end{cor}

  \section{Friezes on closed surfaces}
  \label{closed}

In this section, we present partial results concerning friezes on closed surfaces, as well as examples showing that some results of the previous section do not hold in these settings.

    \begin{prop}
      \label{coprime}
Let $F$ be a frieze on $S$, and let $T$ be an ideal triangulation of $S$ such that for any triangle in $T$ the values of $F$ on the sides of the triangle are mutually coprime. Then $F$ is unitary.

      \end{prop}

      \begin{proof}
We want to represent $S$ as a quotient of a domain $\Omega\subset \H^2$ with preimages of all punctures of $S$ being rational points, such that for any $\gamma\in T$ and any its lift $\tilde\gamma$ in $\Omega$ one has $F(\gamma)=\lambda_{\tilde\gamma}$ with respect to the Ford circles. Then the same argument as in the proof of Theorem~\ref{main} shows that $F$ is unitary.  
        
First, consider a triangle $PQR$ in $T$ with $F(PQ)=r$, $F(PR)=q$, and $F(QR)=a$. We may assume that $\tilde R=\infty=\frac10\in\partial \H^2$. Since all $a,q,r$ are mutually coprime, there exist integers $x,y$ coprime with $a$ and $q$ respectively such that $|ya-xq|=r$. This implies that we can take $\tilde Q=\frac{x}a$ and $\tilde P=\frac{y}q$ with Ford circles as horocycles, see Fig.~\ref{fig-coprime}.

\begin{figure}[!h]
\begin{center}
  \epsfig{file=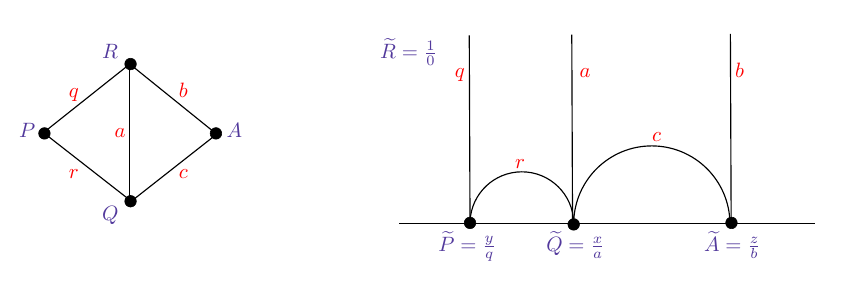,width=0.99\linewidth}
  \caption{To the proof of Prop.~\ref{coprime}: adjacent triangles in triangulation $T$ (left) and their preimage in the universal cover (right).}
\label{fig-coprime}
\end{center}
\end{figure}

Now, consider an adjacent triangle $AQR$ in $T$ with $F(AQ)=c$, and $F(AR)=b$. We claim that $z$ defined by $za-xb=c$ is integer and coprime with $b$, this would imply that we can take $\tilde A=\frac{z}b$ with Ford circle as the horocycle.

Note that we have $y=\frac{xq-r}a\in\Z$. Observe that by Ptolemy relation the $\lambda$-length of the arc $PA$ is equal to
$$\lambda_{PA}=\frac{br+qc}a=\frac{qc+qxb}a-\frac{qxb-br}a=qz-by\in\Z,$$ which implies that
$qz\in\Z$. Since $q$ and $a$ are coprime and $z=\frac{c+xb}a$, we see that $z\in\Z$. Further, $z$ is coprime with $b$ as otherwise $c$ is not coprime with $b$.

The procedure above shows that if a triangle in $T$ is lifted to a triangle in $\H^2$ with rational coordinates and Ford circles as horocycles, then  (in the assumptions of the proposition) every adjacent triangle is also lifted in this way. Therefore, every preimage of every puncture under the quotient map from $\H^2$ to $S$ is rational, and all horocycles are images of Ford circles as required.

        \end{proof}

The following lemma is needed to construct several examples of non-unitary friezes on a four times punctured torus. 

        \begin{lemma}
          \label{ex}
          Let $S$ be a four times punctured torus, consider a unitary triangulation $T$ of $S$ as shown in Fig.~\ref{fig-ex} (we assume all arcs in $T$ to be tagged plain). Then the $\lambda$-length of any loop tagged plain at $P_2$ is divisible by $4$.
          
        \end{lemma}
        
          \begin{proof}
            For any given loop at $P_2$ there exists a finite covering $\tilde S$ of $S$ such that the lift of the loop has two distinct endpoints.  Let $\gamma$ be a loop at $P_2$, and let $\tilde \gamma$ be the lift of $\gamma$ on $\tilde S$ connecting points $\tilde P_2^{(1)}$ and $\tilde P_2^{(2)}$. We need to show that $\lambda_{\tilde\gamma}$ is divisible by $4$.

            Denote by $\tilde T$ the lift of $T$ to $\tilde S$, and assume that $\tilde \gamma$ intersects the interior of $m$ triangles of $\tilde T$. The proof is by induction on $m$. The minimal possible value of $m$ is $4$, and for $m=4,5,6$ the direct calculation shows that the $\lambda$-lengths are equal to $4,8,12$ respectively (see Fig.~\ref{fig-ex}), so we now assume that $m>6$.

            Denote by $\Delta_1,\dots,\Delta_m$ the triangles intersected by $\tilde \gamma$, with $\tilde P_2^{(1)}\in\Delta_1$, and consider the maximal $k<m$ such that a lift of $P_2$ is a vertex of $\Delta_k$ (it is easy to see that such $k>1$ does exist), denote by $\tilde P_2^{(3)}$ that lift of $P_2$. Choose any arc of $\tilde T$ with one endpoint at $\tilde P_2^{(3)}$ not intersecting $\tilde\gamma$, denote by $\tilde A$ the other endpoint of this arc.

            Consider now the quadrilateral $\tilde P_2^{(1)}\tilde P_2^{(2)}\tilde P_2^{(3)}\tilde A$ (or $\tilde P_2^{(1)}\tilde P_2^{(2)}\tilde A\tilde P_2^{(3)}$ depending on the position of $\tilde A$, see Fig.~\ref{fig-ex}) with one side $\tilde\gamma$. By the induction assumption, $\lambda_{\tilde P_2^{(1)}\tilde P_2^{(3)}}$ and $\lambda_{\tilde P_2^{(2)}\tilde P_2^{(3)}}$ are divisible by $4$. By Ptolemy relation,  $\lambda_{\tilde P_2^{(1)}\tilde P_2^{(2)}}\lambda_{\tilde P_2^{(3)}\tilde A}$ is also divisible by $4$, but $\lambda_{\tilde P_2^{(3)}\tilde A}=1$ which proves the lemma. 

          \end{proof}
          
    \begin{figure}[!h]
\begin{center}
  \epsfig{file=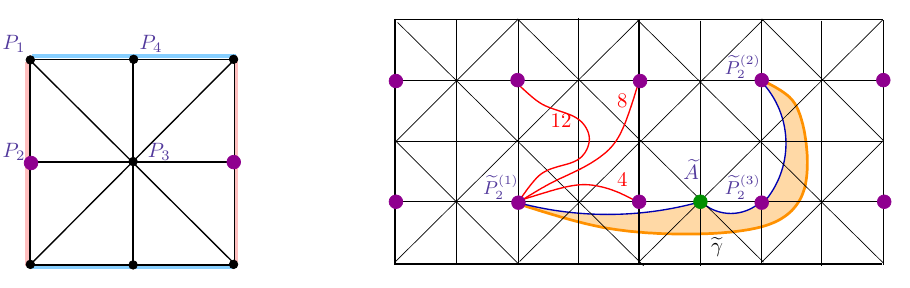,width=0.8\linewidth}
  \caption{Triangulation of a torus with four punctures (left); to the proof of Lemma~\ref{ex} (right). }
\label{fig-ex}
\end{center}
\end{figure}       

Lemma~\ref{ex} gives rise to the following several examples of friezes.

       \begin{example}
          \label{much}
          Let $S$ be a four times punctured torus, and consider a unitary triangulation $T$ of $S$ as in Lemma~\ref{ex} (see Fig.~\ref{fig-ex}). Observe that $\val_T(P_1)=\val_T(P_3)=8$ and $\val_T(P_2)=\val_T(P_4)=4$. Define $k_1=k_3=k_4=2$ and $k_2=8$, i.e. $k_2$ does not divide $\val_T(P_2)$. We claim that the expression
          $$F(\gamma)=\prod_i\!k_i^{\varepsilon_i(\gamma)}\lambda_\gamma$$
          defines nevertheless a frieze on $S$ (as usual, we assume that all arcs of $T$ are tagged plain).

          Indeed, by Theorem~\ref{all-pts}, taking $k_1'=k_3'=k_4'=1$ and $k_2'=4$ would provide a frieze (denote it by $F'$), and taking $k_1''=k_3''=k_4''=4$ and $k_2''=4$ would also provide a frieze (denote it by $F''$). If  $\varepsilon_2(\gamma)\ge -1$ and the other end of $\gamma$ is tagged plain, then $F(\gamma)$ is an integer multiple of $F'(\gamma)$ and thus is integer itself. Similarly, if  $\varepsilon_2(\gamma)\ge -1$ and the other end is tagged notched, then $F(\gamma)$ is an integer multiple of $F''(\gamma)$ and thus is integer.

          Finally, if $\varepsilon_2(\gamma)=-2$ (i.e. $\gamma$ is a loop tagged notched at $P_2$), then $F(\gamma)=\lambda_\gamma/64$. Denote by $\gamma'$ the arc isotopic to $\gamma$ tagged plain at $P_2$. Since $\val_T(P_2)=4$, $$\lambda_\gamma=(\val_T(P_2))^{|\varepsilon_2(\gamma)|}\lambda_{\gamma'}=16\lambda_{\gamma'}.$$ By Lemma~\ref{ex}, $\lambda_{\gamma'}$ is divisible by $4$, and thus $\lambda_{\gamma}$ is divisible by $64$, so $F(\gamma)\in\Z$ as required. 
          
          \end{example}

Example~\ref{much} shows that for closed surfaces the scaling constants $k_i$ is not necessary a divisor of the valence of $P_i$. The next example shows that $k_i$ is not always an integer.

        \begin{example}
          \label{non-integers}
 Again, let $S$ be a four times punctured torus with a unitary triangulation $T$ of $S$ as in Lemma~\ref{ex}. Define $k_1=k_3=k_4=2$ and $k_2=\frac12$, i.e. $k_2\notin\Z$. Then the expression
          $$F(\gamma)=\prod_i\!k_i^{\varepsilon_i(\gamma)}\lambda_\gamma$$
          still defines a frieze on $S$.

The proof is similar to the one in Example~\ref{much}. In this case we should take $k_1'=k_2'=k_3'=k_4'=1$ for $F'$, and $k_1''=k_3''=k_4''=4$, $k_2''=1$ for $F''$, and compare $F$ to them for $\varepsilon_2(\gamma)\le 1$. The values of $F$ on the loops tagged plain at $P_2$ are integer by Lemma~\ref{ex}.            
\end{example}

\begin{figure}[!h]
\begin{center} 
 \epsfig{file=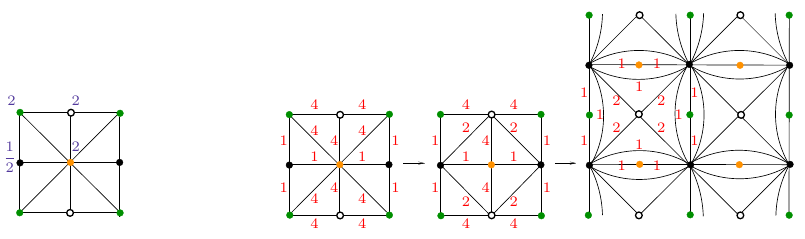,width=1.01\linewidth}
  \caption{To Examples~\ref{non-integers} and~\ref{non-unique}. The scaling constants $\{k_i\}$ (left), and obtaining the same frieze from a different unitary triangulation (right).}
\label{fig-nonu}
\end{center}
\end{figure}

          The next example shows that the counterpart of Remark~\ref{uniqueness} does not hold for closed surfaces: the same frieze can appear as a rescaling of two distinct unitary friezes.

          \begin{example}
\label{non-unique}

Consider the frieze defined in Example~\ref{non-integers}. Applying a sequence of flips and applying Prolemy relation, we obtain triangulation $T'$ shown in Fig.~\ref{fig-nonu}. In this triangulation, $F(\gamma)=1$ for all arcs of $T'$ not incident to $P_4$, $F(\gamma)=2$ for all arcs of $T'$ incident to $P_4$, and $\val_{T'}(P_4)=4$. Therefore, $F$ can be obtained from a unitary frieze $\hat F'$ with unitary triangulation $T'$ as  $F(\gamma)=\prod_i{\hat k_i}^{\varepsilon_i(\gamma)}\hat F'(\gamma)$, where $\hat k_1=\hat k_2=\hat k_3=1$ and $\hat k_4=2$.  

\end{example}

\appendix
\section{Unimodular matrices in Conway--Coxeter friezes}

The appendix is devoted to a geometric answer to the question of Alain Valette. 

\begin{prop}
  \label{every}
For every matrix $M=\begin{pmatrix}a&b\\ c&d\end{pmatrix}\in SL_2(\Z)$ with positive entries there exists a Conway--Coxeter frieze containing $M$ as a diamond\qquad    
{\small
$
  \begin{tabular}{ccccccccccccccccccccc}
&$b$\\
    $a$&&$d$\\
    &$c$\\
  \end{tabular}.
$
}
  \end{prop}

\begin{proof}
Consider the quadrilateral  $ABCD\subset\H^2$ with rational vertices, where $A=\frac01$, $B=\frac10$, $C=\frac{a}b$  and   $D=\frac{c}d$. Since $M\in SL_2(\Z)$, the expressions for $C$ and $D$ are reduced, and $\frac{a}b>\frac{c}d$. Computing the $\lambda$-lengths with respect to Ford circles by using the  equality~(\ref{eq_det}), it is easy to see that the diagonals of $ABCD$ have $\lambda$-lengths $a,d$, and the $\lambda$-lengths of sides are $1,b,1,c$ (listed in a cyclic order), see Fig.~\ref{fig-abcd}.

Now consider the part of the Farey graph spanned by $A,B,C,D$ and all vertices of all triangles of the Farey triangulation intersected by the lines $AD$ and $BC$. The convex hull of all these points is a polygon $\mathcal{P}$ with sides being edges of the Farey graph, so $\lambda$-lengths of all sides of $\mathcal{P}$ are equal to $1$ (note that $AB$ and $CD$ are sides of $\mathcal{P}$). Therefore, the Farey triangulation restricted to $\mathcal{P}$ defines a Conway--Coxeter frieze containing the diamond with entries $a,b,c,d$.    

\end{proof}

\begin{figure}[!h]
\begin{center}
 \epsfig{file=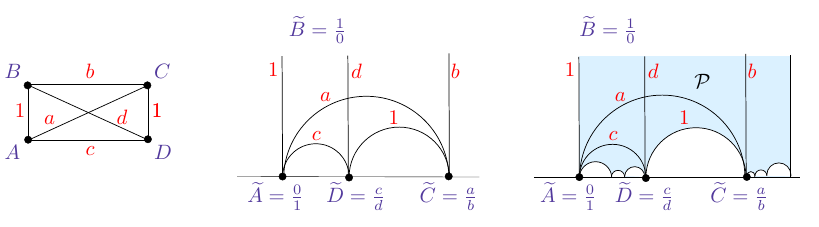,width=0.99\linewidth}
  \caption{Embedding the polygon $ABCD$ into an ideal polygon with sides contained in the Farey graph.}
\label{fig-abcd}
\end{center}
\end{figure}

It was noted to the authors by Aleksei Ustinov that the statement can be easily proved using continued fractions. 

The geometric proof can also be extended to certain unions of diamonds.

\end{document}